\documentclass[10pt]{amsart}

\usepackage{amsfonts,amssymb,amsmath,amsthm,amscd,enumerate}
\usepackage{latexsym}
\usepackage{euscript}
\usepackage[english]{babel}
\usepackage[latin1]{inputenc}

\newtheorem{theorem}{Theorem}[section]
\newtheorem{definition}[theorem]{Definition}
\newtheorem{lemma}[theorem]{Lemma}
\newtheorem{corollary}[theorem]{Corollary}

\newtheorem{Remark}[theorem]{Remark}

\title[Matrices over an infinite integral domain]{Graded polynomial identities and central polynomials of matrices over an infinite
integral domain}
\author{Luís Felipe Gonçalves Fonseca}
\address{Departamento de Matem\'atica, Universidade Federal
de Viçosa - Campus Florestal, 35690-000 - Florestal}
\email{luisfelipe@ufv.br}

\begin{document}
\maketitle

\begin{abstract}
Let $K$ be an infinite integral domain and $M_{n}(K)$ be the algebra
of all $n\times n$ matrices over $K$. This paper aims for the
following goals:
\begin{itemize}
\item Find a basis for the graded identities for elementary grading in $M_{n}(K)$
when the neutral component and diagonal coincide; \item Describe the
$\mathbb{Z}_{p}$-graded central polynomials of $M_{p}(K)$ when $p$
is a prime number; \item Describe the $\mathbb{Z}$-graded central
polynomials of $M_{n}(K)$.
\end{itemize}
\end{abstract}

\section{Introduction}

Polynomial Identity theory (PI) is an important branch of the Ring
Theory. The first crucial developments in PI-theory were Kaplansky's
Theorem \cite{Kaplansky} about primitive PI-algebras and the
Amitsur-Levitsky Theorem \cite{Amitsur}, published in 1948 and 1950,
respectively. The latter theorem is important for describing the
polynomial identities of matrices.

Let $K$ be a field and $M_{n}(K)$ be the algebra of all $n\times n$
matrices over $K$. PI-theory is used to obtain a basis for
polynomial identities of $M_{n}(K)$. Razmyslov \cite{Razmyslov2}
discovered a nine-polynomial basis for the identities of $M_{2}(K)$
when $K$ is a field of characteristic zero. Some years later,
Drensky \cite{Drensky1} found a minimal polynomial basis: comprising
the Hall identity and the standard polynomial of degree $4$.
Koshlukov \cite{Plamen1} found a basis (consisting of four
identities) for the identities of $M_{2}(K)$ when $K$ is an infinite
field of $char K = p > 2$.

Despite these advances, the identities of $M_{2}(K)$ when $K$ is an
infinite field of characteristic $2$ or an infinite integral domain
remain unresolved.

Let $K$ be a field of characteristic zero. In 1950 Specht
\cite{Specht} conjectured that every system of identities in
associative algebra has a finite basis. Specht's conjecture was
unsolved until the late 1980s when Kemer demonstrated its truth
using the theory of $\mathbb{Z}_{2}$-graded algebras.

Another important problem is describing the graded polynomial
identities of $M_{n}(K)$. The $\mathbb{Z}_{2}$-graded polynomial
identities of $M_{2}(K)$ were described by Di Vincenzo
\cite{Vincenzo}, while Vasilovsky \cite{Vasilovsky 1} described the
$\mathbb{Z}_{n}$-graded polynomial identities of $M_{n}(K)$. A year
earlier, the same author had described the $\mathbb{Z}$-graded
identities of $M_{n}(K)$ \cite{Vasilovsky 2}. Vasilovsky's results
were extended by Bahturin and Drensky \cite{Drensky2}, who found the
basis of the graded identities for the elementary gradings on
$M_{n}(K)$ when the neutral component and diagonal of $M_{n}(K)$
coincide. Azevedo \cite{Azevedo1},\cite{Azevedo2} and Silva
\cite{Silva} extended the Vasilovsky's and Bahturin-Drensky's
results, respectively to an infinite field.

Kaplansky \cite{Kaplansky2} posed a list of open problems in Ring
Theory. Among these was the question does a non-trivial contral
polynomial exist for $M_{n}(K)$ when $n \geq 3$? This question was
answered by Formanek \cite{Formanek}, and independently by Razmyslov
\cite{Razmyslov3}.

Describing the central polynomials of $M_{n}(K)$ is a crucial task
in PI-theory. When $K$ is a field of characteristic zero, a set of
generators may be found for the central polynomials of $M_{2}(K)$
\cite{Okhitin}. Koshlukov and Colombo \cite{Plamen3} described the
central polynomials of $M_{2}(K)$, when $K$ is an infinite field of
characteristic $p > 2$.

The first attempts at describing graded polynomials of $M_{n}(K)$
were made by Brandão Júnior \cite{Junior}. Assuming an infinite
ground field $K$, he described the $\mathbb{Z}_{n}$-graded central
polynomials of $M_{n}(K)$ when $char K \nmid n$, as well as, the
$\mathbb{Z}_{p}$-graded central polynomials of $M_{p}(K)$ (where
$char K = p > 2$) and the $\mathbb{Z}$-graded central polynomials of
$M_{n}(K)$.

Few reports of graded identities of $M_{n}(K)$ exist in the
literature, when $K$ is an infinite integral domain. Brandão Júnior,
Koshlukov and Krasilnikov \cite{B} detailed a basis for the
$\mathbb{Z}_{2}$-graded identities of $M_{2}(K)$. They also
described a basis for the $\mathbb{Z}_{2}$-graded central
polynomials of $M_{2}(K)$.

In this paper, we combine the methods of \cite{Azevedo1},
\cite{Azevedo2}, \cite{Drensky2}, \cite{Junior}, \cite{Fonseca},
\cite{Silva}, \cite{Vasilovsky 1} and \cite{Vasilovsky 2}. This
paper aims for the following goals:
\begin{itemize}
\item Find a basis for the graded identities for elementary grading in $M_{n}(K)$
when the neutral component and diagonal coincide; \item Describe the
$\mathbb{Z}_{p}$-graded central polynomials of $M_{p}(K)$ when $p$
is a prime number; \item Describe the $\mathbb{Z}$-graded central
polynomials of $M_{n}(K)$.
\end{itemize}
In these three situations, $K$ is an infinite integral domain.

\section{Preliminaries}


Let $K$ be a fixed unital associative and commutative ring . We
assume that all modules are left-modules over $K$ and all (unital
associative) algebras are considered over $K$. We assume that all
ideals are bilateral ideals. The set $\{1,\cdots,n\}$ is denoted by
$\widehat{n}$. Moreover, $G$ denotes an arbitrary group and
$\mathbb{N} = \{0,1,2,\cdots,n,\cdots\}$ is the set of natural
numbers. Here, $S_{n}$ denotes the group of permutations on
$\widehat{n}$. $H_{n}$ denotes the subgroup of $S_{n}$ generated by
$(1 2 \cdots n)$.

Let $X$ be a countable set of variables and $K\langle X \rangle$ be
the free associative ring freely generated by $X$. Let $A$ be an
algebra over $K$ and let $Z(A)$ be the center of $A$. A polynomial
$f(x_{1},\cdots,x_{n}) \in K\langle X \rangle$ is called an ordinary
polynomial identity (respectively an ordinary central polynomial)
for $A$ if $f(a_{1},\cdots,a_{n}) = 0$ for all $a_{1},\cdots,a_{n}
\in A$ (respectively $f(0,\cdots,0) = 0$ and $f(a_{1},\cdots,a_{n})
\in Z(A)$). The algebra $A$ is a PI-algebra if there exists $f \in
K\langle X \rangle$ satisfying the following conditions:
\begin{description}
\item $f$ is an ordinary polynomial identity for $A$;
\item Some coefficient in the highest-degree
homogeneous component of $f$ equals to $1$.
\end{description}

The set of ordinary polynomial identities (respectively ordinary
central polynomials) of $A$ is an ideal (respectively submodule) of
$K\langle X \rangle$ that is invariant under all endomorphisms of
$K\langle X \rangle$. The ideals (respectively submodules) of
$K\langle X \rangle$ that are invariant under all endomorphisms of
$K\langle X \rangle$ are called $T$-ideals (respectively
$T$-spaces).

Clearly, the intersection of a family of $T$-ideals (respectively
$T$-spaces) of $K\langle X \rangle$ is also a $T$-ideal
(respectively a $T$-space). Let $S$ be a non-empty set of $K\langle
X \rangle$. The $T$-ideal (respectively $T$-space) generated by $S$,
denoted $\langle S \rangle_{T}$ (respectively $\langle S
\rangle^{T}$), is the intersection of all $T$-ideals (respectively
$T$-spaces) containing $S$.

An algebra $A$ over $K$ is $G$-graded when there exist
$K$-submodules $\{A_{g}\}_{g \in G} \subset A$ such that:
\begin{center}
$A = \bigoplus_{g\in G}A_{g} \ \ (1);$ \\
$A_{g}A_{h} \subset A_{gh} \ \mbox{for all} \ g,h \ \in G \ \ (2)$.
\end{center}

Each submodule $A_{g}$ is called the homogeneous component of
$G$-degree $g$ and its non-zero elements are homogeneous elements of
$G$-degree $g$. Moreover, if $a$ is a homogeneous element, its
$G$-degree is denoted by $\alpha(a)$. We denote the identity element
of $G$ by $e$. The support of $A$, with respect to the grading
$\{A_{g}\}_{g \in G}$, is the following subset of $G$:
\begin{center}
$Supp_{G}(A):= \{g \in G | A_{g} \neq \{0\}\}$.
\end{center}

Let $\{X_{g}| g \in G\}$ be a family of disjoint countable sets
indexed by $G$ and let $X = \bigcup_{g \in G} X_{g}$. $K\langle X
\rangle_{g}$ is the $K$-submodule of $K\langle X \rangle$ spanned by
$m = x_{j_{1}}\cdots x_{j_{k}}$ such that $\alpha(m) = g$. The
decomposition $K\langle X \rangle = \bigoplus_{g \in G} K\langle X
\rangle_{g}$ is a $G$-grading, whereby $K\langle X \rangle$ is the
$G$-graded free associative ring freely generated by $X$. A monomial
is a variable or a product of variables in $X$.

Let $m = x_{i_{1}}\cdots x_{i_{l}}$ be a monomial of $K\langle X
\rangle$. We denote by $h(m)$ the $l$-tuple
$(\alpha(x_{i_{1}}),\cdots,\alpha(x_{i_{l}}))$.

An endomorphism $\phi$ of $K\langle X \rangle$ is called $G$-graded
endomorphism when $\phi(K\langle X \rangle_{g}) \subset K\langle X
\rangle_{g} \ \forall \ g \in G$. When a graded ideal (respectively
a graded submodule) $I \subset K\langle X \rangle$ is invariant
under all $G$-graded endomorphisms of $K\langle X \rangle$ is called
a $T_{G}$-ideal (respectively a $T_{G}$-space). A graded polynomial
$f(x_{1},\cdots,x_{n}) \in K\langle X \rangle$ is a $G$-graded
polynomial identity for $A$ (respectively a $G$-graded central
polynomial for $A$) if $f(a_{1},\cdots,a_{n}) = 0$ for all $a_{i}
\in A_{\alpha(x_{i})}, \ i = 1,\cdots,n$ ($f(0,\cdots,0) = 0$ and
$f(a_{1},\cdots,a_{n}) \in Z(A)$ for all $a_{i} \in
A_{\alpha(x_{i})}, \ i = 1,\cdots,n$).

The set of all $G$-graded identities of $A$ (respectively all
$G$-graded central polynomials of $A$) is denoted by $T_{G}(A)$
(respectively $C_{G}(A)$). Clealy, $T_{G}(A)$ (respectively
$C_{G}(A)$) is a $T_{G}$-ideal (respectively a $T_{G}$-space and a
subalgebra) and the intersection of a family of $T_{G}$-ideals
(respectively $T_{G}$-spaces) of $K\langle X \rangle$ is also a
$T_{G}$-ideal (respectively a $T_{G}$-space). The $T_{G}$-ideal
(respectively $T_{G}$-space) generated by a non-empty set $S$ of
$K\langle X \rangle$ is defined as in the ordinary case. The
$T_{G}$-ideal (respectively a $T_{G}$-space) generated by $S$ is
denoted by $\langle S \rangle_{T_{G}}$ (respectively $\langle S
\rangle^{T_{G}}$). A graded polynomial $f$ is said to be a
consequence of $S \subset K\langle X \rangle$ if $f \in \langle S
\rangle_{T_{G}}$ (or equivalently, that $f$ follows from $S$). A set
$S \subset K\langle X \rangle$ is called a basis for the graded
identities (respectively the graded central polynomials) of $A$ if
$T_{G}(A) = \langle S \rangle_{T_{G}}$ (respectively $C_{G}(A) =
\langle S \rangle^{T_{G}}$).

The matrix unit $e_{ij} \in M_{n}(K)$ contains $1$ as the only a
non-zero value in the $i$-th row and $j$-th column. Given an
$n$-tuple $\overline{g} = (g_{1},\cdots,g_{n}) \in G^{n}$, a
$G$-grading is determined in $M_{n}(K)$ by stipulating that $e_{ij}$
is homogeneous of $G$-degree $g_{i}^{-1}g_{j}$. These gradings are
elementary and we say that $M_{n}(K)$ possesses an elementary
grading induced by $\overline{g}$. We equipped $M_{n}(K)$ with an
elementary grading induced by an $n$-tuple of distinct elements from
$G$. The set $\{g_{1},\cdots,g_{n}\}$ is denoted by $G_{n}$.

Below we provided two important examples of elementary gradings on
$M_{n}(K)$ whose neutral component and diagonal coincide:

\begin{description}
\item $\mathbb{Z}_{n}$-canonical grading (or $\mathbb{Z}_{n}$-grading): when $G = \mathbb{Z}_{n}$
and the $n$-tuple $\overline{g}$ is
$(\overline{1},\overline{2},\cdots,\overline{n-1},\overline{n})$;
\item $\mathbb{Z}$-canonical grading (or $\mathbb{Z}$-grading): when $G = \mathbb{Z}$
and the $n$-tuple $\overline{g}$ is $(1,2,\cdots,n-1,n)$.
\end{description}



\section{Silva's Generic Model}\label{silva}

Generic models have an important role in PI (see for instance
\cite{Drensky3}, \cite{Drensky4} and \cite{Procesi}). In this
section, we recall Silva's Generic Model, as described in
\cite{Silva}. Let $G$ be an arbitrary group and let $\overline{g} =
(g_{1},\cdots,g_{n}) \in G^{n}$ be an $n$-tuple of distinct elements
from $G$. We consider the algebra $M_{n}(K)$ to be equipped with the
elementary grading induced by $\overline{g}$. For each $h \in G$,
let $Y_{h} = \{y_{h,i}^{k}| 1\leq k \leq n; i \geq 1\}$ be a
countable set of commuting variables and let $Y = \bigcup_{h \in
G}Y_{h}$. Let $\Omega = K[Y]$ denote the polynomial ring with
commuting variables in $Y$. Let $M_{n}(\Omega)$ be the algebra of
all $n\times n$ matrices over $\Omega$. This algebra may be equipped
with the elementary grading induced by $\overline{g}$ as $M_{n}(K)$.
Let $G_{n}$ denote the set $\{g_{1},\cdots,g_{n}\}$.

\begin{definition}\label{def}
Let $h$ be an element of $G$. Let $L_{h}$ denote, the set of all
indices $k \in \widehat{n}$ such that $g_{k}h \in G_{n}$. Let
$s_{h}^{k}$ denote the index determined by $g_{s_{h}^{k}}:=g_{k}h$.

Let $\textbf{h} = (h_{1},\cdots,h_{m}) \in G^{m}$. $L_{\textbf{h}}$
defines (the set of indices associated with the $m$-tuple
$(h_{1},\cdots,h_{m}))$ the subset of $\widehat{m}$ such that its
elements satisfy the following property:
\begin{center}
$g_{k}h_{1}\cdots h_{i} \in G_{n} \ \forall \ i \ \in \widehat{m}$.
\end{center}

We define a sequence $(s_{1}^{k},\cdots,s_{m+1}^{k})$(the sequence
associated with $\textbf{h}$ is determined by $k$) inductively by
setting:
\begin{description}
\item $1) s_{1}^{k} = k$;
\item $2) s_{l}^{k} : g_{s_{l}^{k}} = g_{k}h_{1}\cdots h_{l-1} \ \ \forall \ l \in \{2,\cdots,m+1\}$
\end{description}
\end{definition}

A generic matrix of $G$-degree $h$ is a homogeneous element of
$M_{n}(\Omega)$ the following type:
\begin{center}
$A_{i}^{h} = \sum_{k \in L_{h}} y_{h,i}^{k}e_{k,s_{h}^{k}}$.
\end{center}
The $G$-graded subalgebra of $M_{n}(\Omega)$ generated by the
generic matrices is called the algebra of the generic matrices which
we denote by $R$. Notice that $Supp_{G}(R) = Supp_{G}(M_{n}(K))$.

The next lemma is an important computational result. Its proof is
the immediate consequence of a multiplication table of matrix units.

\begin{lemma}(\cite{Silva}, Lemma 3.5)\label{felipe0}
If $L$ is the set of indices associated with the $q$-tuple
$(h_{1},\cdots,h_{q})$ in $G^{q}$ and $s_{k} = (s_{1}^{k},\cdots,
s_{q+1}^{k})$ denotes the corresponding sequence determined by $k
\in L$ ,then
\begin{center}
$A_{i_{1}}^{h_{1}}\cdots A_{i_{q}}^{h_{q}} = \sum_{k \in
L}w_{k}e_{s_{1}^{k},s_{q+1}^{k}}$
\end{center}
which $w_{k}=
y_{h_{1},i_{1}}^{s_{1}^{k}}y_{h_{2},i_{2}}^{s_{2}^{k}}\cdots
y_{h_{q},i_{q}}^{s_{q}^{k}}$.
\end{lemma}

\begin{definition}
Let $f(x_{1},\cdots,x_{n})$ be a polynomial of $K\langle X \rangle$
and let $A_{1} \in R_{\alpha(x_{1})},\cdots, \newline A_{n} \in
R_{\alpha(x_{n})}$. $f(A_{1},\cdots,A_{n})$ denotes the result of
replacing for the corresponding elements of $R$.
\end{definition}

The next lemma is the same in (\cite{Silva}, Lemma 4.5).

\begin{lemma}\label{felipe5}
Let $M(x_{1},\cdots,x_{q})$ and $N(x_{1},\cdots,x_{q})$ be two
monomials of $K\langle X \rangle$ that start with the same variable.
Let $m(x_{1},\cdots,x_{q})$, $n(x_{1},\cdots,x_{q})$ be two
monomials obtained from $M$ and $N$ respectively by deleting the
first variable. If there exist matrices $A_{1},\cdots,A_{q}$, such
that $M(A_{1},\cdots,A_{q})$ and $N(A_{1},\cdots,A_{q})$ have, in
the same position, the same non-zero entry ,then the matrices
$m(A_{1},\cdots,A_{q})$ and $n(A_{1},\cdots,A_{q})$ also have, in
the same position, the same non-zero entry.
\end{lemma}
\begin{proof}
It is the immediate consequence of Lemma \ref{felipe0}.
\end{proof}

\begin{Remark}(\cite{Silva}, Corollary 3.7)\label{silva4}
Notice that if $m_{1}$ and $m_{2}$ are two monomials such that
$h(m_{1}) = h(m_{2})$, then $m_{1} \in T_{G}(R)$ if and only if
$m_{2} \in T_{G}(R)$.
\end{Remark}

\begin{Remark}\label{silva5}
Let $m$ be a monomial. Notice that $m \in T_{G}(R)$ if and only if
$m \in T_{G}(M_{n}(K))$.
\end{Remark}

The next lemmas three lemmas can be proved by elementary algebraic
methods.

\begin{lemma}\label{felipe3}
Let $f \in T_{G}(R)$. Then all multi-homogeneous components of $f$
are elements of $T_{G}(R)$.
\end{lemma}

\begin{lemma}\label{felipe12}
Let $f \in C_{G}(R)$. Then all multi-homogeneous components of $f$
are elements of $C_{G}(R)$.
\end{lemma}

\begin{lemma}\label{felipe4}
Let $m(x_{1},\cdots,x_{q}) = x_{i_{1}}\cdots x_{i_{r}}$ and
$n(x_{1},\cdots,x_{q})$ be two monomials such that the matrices
$n(A_{1},\cdots,A_{q})$ and $m(A_{1},\cdots, A_{q})$ have, at some
position, the same non-zero entry. Then $m - n$ is a
multi-homogeneous polynomial.
\end{lemma}

Observe that $T_{G}(R) \subset T_{G}(M_{n}(K))$. The proof of the
next lemma is similar to that proof in (\cite{Azevedo1},Lemma 3).

\begin{lemma}\label{felipe2}
Let $K$ be an infinite integral domain.Then $T_{G}(M_{n}(K)) =
T_{G}(R)$.
\end{lemma}

\begin{corollary}\label{felipe11}
Let $K$ be an infinite integral domain.Then $C_{G}(M_{n}(K)) =
C_{G}(R)$.
\end{corollary}

\section{Some graded identities of $R$ and Type 1-monomials}

In this section, we present some graded identities for elementary
grading in $R$ when the neutral component and diagonal coincide.
Notice that $R$ is equipped with the elementary grading induced by
an $n$-tuple of pairwise elements from $G$ if and only if $R_{e}$
and diagonal coincide.

The next lemma was proved by Bahturin and Drensky in
(\cite{Drensky2}, Lemma 4.1) for full algebra of $n$ by $n$ matrices
over a field of characteristic zero.

\begin{lemma}\label{luis1}
The following graded polynomials are $G$-graded polynomial
identities for $R$:
\begin{center}
$\bullet \ x_{1}x_{2} - x_{2}x_{1} \ \ \mbox{when} \ \ \alpha(x_{1})
= \alpha(x_{2}) = e
\ (1)$; \\

$\bullet \ x_{1}x_{2}x_{3} - x_{3}x_{2}x_{1} \ \ \mbox{when} \ \
\alpha(x_{1}) = \alpha(x_{3})
= (\alpha(x_{2}))^{-1} \neq e \ \ (2)$; \\

$\bullet \ x_{1} \ \ \mbox{when} \ \ R_{\alpha(x_{1})} = \{0\} \ \
(3)$.
\end{center}
\end{lemma}
\begin{proof}
It follows from Lemma \ref{felipe0} and the proof of Lemma 4.1 in
\cite{Drensky2}.
\end{proof}

\begin{definition}
Let $J$ be the $T_{G}$-ideal generated by $(1),(2)$ and $(3)$. Let
$J_{1}$ be the $T_{G}$-ideal generated by $(1)$ and $(2)$ only.
\end{definition}

\begin{definition}
Let $m = x_{i_{1}}\cdots x_{i_{q}}$ be a $G$-graded monomial. An
element $n \in K\langle X \rangle$ is called a subword of $m$ when
there exist $j \in \{0,\cdots,q\}$ and $l \in \mathbb{N}$, where $j
+ l \leq q$, such that:
\begin{center}
$n = x_{i_{j}}\cdots x_{i_{j+l}}$.
\end{center}
Likewise, the monomial $n$ is termed a proper subword of $m$ when
$n$ is a subword of $m$, $n\neq m$ and $n \neq 1$.
\end{definition}

\begin{definition}
Let $m = x_{i_{1}}\cdots x_{i_{l}}$ be a $G$-graded monomial. This
monomial is called a Type 1-monomial when the $G$-degree of all its
non-empty subwords are elements of $Supp_{G}(R)$.
\end{definition}

(\cite{Drensky2}) shows an example of $G$-graded Type 1-monomial
identities of $M_{n}(K)$ (see Example 4.7,\cite{Drensky2}).

In this paper, the following lemma is useful.

\begin{lemma}\label{felipe6.999}
Let $m$ be a multilinear Type 1-monomial. Let $\overline{m}$ be the
polynomial obtained from $m$ by deleting the variables of $G$-degree
$e$ by one. Then $m \in T_{G}(R)$ if and only if $\overline{m} \in
T_{G}(R)$.
\end{lemma}

The proof of the following lemma is left as an exercise.

\begin{lemma}\label{luis4}
Let $m \in T_{G}(R)$ be a monomial. If $m$ is not a Type 1-monomial
identity, then $m$ follows from $(3)$.
\end{lemma}

\section{Type 1-monomial identities of $R$}

This section describes the monomial Type 1-identities  for
elementary grading in $R$ when the neutral component and diagonal
coincide. The number $s$ denotes $|Supp_{G}(R)|$ and $\lambda$
denotes the number $[s+1][(s+1)(\sum_{i=1}^{s} (s - 1)^{i}) + 1]$.

\begin{definition}
Let $m = x_{i_{1}}\cdots x_{i_{q}}$. Let $k,l$ be two positive
integers such that $1 \leq k \leq l \leq q$. We define the monomial
$m^{[k,l]}$ obtained from $m$ by deleting the $k-1$ first variables
and the $q - l$ last variables.
\end{definition}

\begin{definition}
Let $\mathcal{S}$ denote the set of all sequences with elements in
$Supp_{G}(R)$ whose length is less than $s + 1$. Let $\mathcal{A} =
\{(g_{1},\dots,g_{m}) \in \mathcal{S}| g_{1}.\dots.g_{m} = e\}$.
\end{definition}
\begin{Remark}
Notice that $|\mathcal{S}| = \sum_{i=1}^{s} s^{i}$.
\end{Remark}
\begin{definition}
A monomial $m = x_{i_{1}}\cdots x_{i_{l}}$ is called a Type
2-monomial when there exist $a \in \mathbb{N} - \{0\}$, $p_{1},p_{2}
\in \widehat{l}$ such that $1 \leq p_{1} < p_{1} + a < p_{2} < p_{2}
+ a \leq l$ and:
\begin{description}
\item $\alpha(x_{i_{p_{1}}}\cdots x_{i_{p_{1} + a}}) =
\alpha(x_{i_{p_{1}+a+1}}\cdots x_{i_{p_{2}-1}}) =
\alpha(x_{i_{p_{2}}}\cdots x_{i_{p_{2} + a}}) = e$;
\item $ h(x_{i_{p_{1}}}\cdots x_{i_{p_{1} + a}}) =
h(x_{i_{p_{2}}}\cdots x_{i_{p_{2} + a}})$.
\end{description}
\end{definition}

\begin{definition}
Let $m = x_{i_{1}}\cdots x_{i_{l}}$ be a monomial. It is called a
Type 3-monomial when it does not have a proper subword of $G$-degree
$e$. Otherwise, it is called a Type 4-monomial.
\end{definition}

\begin{corollary}\label{tec2}
Let $m = x_{i_{1}}\cdots x_{i_{l}}$ be a Type 1-monomial without
variables of $G$-degree $e$. If $l > s$, then $m$ is a Type 4-
monomial.
\end{corollary}
\begin{proof}
Let $\beta(t) =\alpha(m^{[1,t]})$ be a function with domain
$\widehat{l}$ and codomain $Supp_{G}(R)$. For the hypothesis, $l
> s$. Consequently, according to the Pigeonhole Principle, there
exists $1 \leq t_{1} < t_{2}\leq l$ such that $\beta(t_{1}) =
\beta(t_{2})$. Note that $t_{1} + 1 < t_{2}$ because $m$ does not
have variable of $G$-degree $e$. So, $m^{[t_{1}+1,t_{2}]}$ satisfies
the thesis statement of the corollary.
\end{proof}

\begin{lemma}\label{tec4}
Let $S$ be a multiset formed by elements of
$\widehat{(\sum_{i=1}^{s} (s - 1)^{i})}$. If $|S| \geq (s +
1)(\sum_{i=1}^{s} (s - 1)^{i}) + 1$, then there exists $i \in
\widehat{(\sum_{i=1}^{s} (s - 1)^{i})}$ such that this positive
integer repeats, at least, $s + 2$ times in $S$.
\end{lemma}
\begin{proof}
It is the immediate consequence of the Pigeonhole Principle.
\end{proof}

\begin{lemma}\label{felipe9.42}
Let $m = x_{1}\cdots x_{r}$ be a multilinear Type 1-monomial without
variables of $G$-degree $e$. If the ordinary degree of $m$ is
greater than or equal to $\lambda$, then it is a Type 2-monomial.
\end{lemma}
\begin{proof}
Let $m$ be a multilinear Type 1-monomial of $K\langle X \rangle$
without variables of $G$-degree $e$ whose ordinary degree is greater
than or equal to $(s+1)^{2}(\sum_{i=1}^{s} (s - 1)^{i}) + (s + 1)$
and $a$ denotes the number $(s+1)(\sum_{i=1}^{s} (s - 1)^{i}) + 1$.

Let:
\begin{center}
$m_{1} = m^{[1,s+1]}, m_{2} = m^{[s+2, 2(s + 1)]}, \cdots , m_{a} =
m^{[(a-1)(s+1) + 1,a.(s+1)]}$.
\end{center}
By Corollary \ref{tec2}, for each $m_{i}$, there is a proper subword
of $G$-degree $e$ and ordinary degree less than or equal to $s$.

Let $\gamma: \{m_{1},\cdots,m_{a}\} \rightarrow \mathcal{A}$ be a
relation that assigns: $(g_{1},\cdots,g_{n_{1}}) \in \gamma(m_{i})$
if, and only if, there exists a subword of $m_{i}$ of ordinary
degree $n_{1}$, $m_{i,1}$, such that $h(m_{i,1}) =
(g_{1},\cdots,g_{n_{1}})$. By Lemma \ref{tec4}, there exist subwords
$m_{i_{1},1},\cdots, m_{i_{s+2},1} \in \newline
\{\gamma(m_{1}),\cdots,\gamma(m_{a})\}$ of $m_{i_{1}},\cdots,
m_{i_{s+2}} (i_{1} < \cdots < i_{s+2})$ such that $h(m_{i_{1},1}) =
\cdots = h(m_{i_{s+2},1})$. By Pigeonhole Principle, there exist
$k,k+l \in \{1,\cdots,s+2\}$ such that the subword of $m$
($m_{i_{k},i_{k+l},1}$), between $m_{i_{k},1}$ and $m_{i_{k+l},1}$,
with $G$-degree $e$.

Therefore, $m_{i_{k},1}m_{i_{k},i_{k+l},1}m_{i_{k+l},1}$ is a Type
2- monomial. So, $m$ is Type 2-monomial as well.
\end{proof}

\begin{definition}
We denote by $U$ the $T_{G}$-ideal generated by the following
identities of $R$:
\begin{center}
$\bullet \ x_{1}x_{2} - x_{2}x_{1} \ \ \mbox{when} \ \ \alpha(x_{1})
= \alpha(x_{2}) = e \
\ (1)$; \\

$\bullet \ x_{1}x_{2}x_{3} - x_{3}x_{2}x_{1} \ \ \mbox{when} \ \
\alpha(x_{1}) = \alpha(x_{3})
= (\alpha(x_{2}))^{-1} \neq e \ \ (2)$; \\

$\bullet \ x_{1} \ \ \mbox{when} \ \ R_{\alpha(x_{1})} = \{0\} \ \ (3)$; \\

$\bullet \ $ The multilinear Type 1- monomial identities whose
ordinary degrees are less than or equal to $\lambda$ \ \ $(4)$.
\end{center}
\end{definition}

Recall that if $m$ is a monomial, then $m \in T_{G}(R)$ if and only
if $m \in T_{G}(M_{n}(K))$ (Remark \ref{silva5}). In the next lemma
we follow an idea of (\cite{Drensky2}, Proposition 4.2).

\begin{lemma}\label{felipe9.45}
Let $m = x_{1}\cdots x_{q}$ be ($q > \lambda$) a multilinear
monomial. If $m$ is a Type 1- monomial identity for $R$, then $m$
follows from $(4)$.
\end{lemma}
\begin{proof}
Let $m = x_{1}\cdots x_{q}$ be a multilinear Type 1-monomial
identity for $R$ where $q \geq \lambda + 1$. We may suppose without
any loss of generality that $\alpha(x_{i}) \neq e$ for all $i \in
\widehat{q}$ (Lemma \ref{felipe6.999}). The proof is made by
induction on $q$. Suppose that $q = \lambda + 1$. According to Lemma
\ref{felipe9.42}, $m$ is a Type 2-monomial.

Here, we use the same notation as in Lemma \ref{felipe9.42}. If
$x_{p_{1}}\cdots x_{p_{2} + a}$ is a graded monomial identity for
$R$, then $m$ is a consequence of the monomial identities of type
$(4)$. Thus, we may suppose that $x_{p_{1}}\cdots x_{p_{2} + a}
\notin T_{G}(R)$. Let $\widehat{m} = m^{[1,p_{1}-1]}m^{[p_{1} + a +
1,q]}$. If $\widehat{m}$ is a monomial identity for $R$, then $m$ is
a consequence of $\widehat{m}$ that is a Type 1-monomial. 
Suppose for contradiction that $\widehat{m}$ is not a polynomial
identity for $R$. We may suppose without loss of generality that
$p_{1} + a + 1 < p_{2}$ and $q \geq p_{2} + a + 1$. Therefore, there
exist the matrix units $e_{l_{1}k_{1}} \in
M_{n}(K)_{\alpha(x_{1})},\cdots, e_{l_{p_{1} -1}k_{p_{1} - 1}} \in
M_{n}(K)_{\alpha(x_{p_{1} - 1})}, e_{l_{p_{1} + a + 1}k_{p_{1} + a +
1}}\in M_{n}(K)_{\alpha(x_{p_{1} + a + 1})},\cdots, e_{l_{q}k_{q}}
\in M_{n}(K)_{\alpha(x_{q})}$ such that $(e_{l_{1}k_{1}}\cdots
e_{l_{p_{1} -1}k_{p_{1} -
1}}).(e_{l_{p_{1}+a+1}k_{p_{1}+a+1}}.\cdots.e_{l_{q}k_{q}}) \neq 0$.

Note that $k_{p_{1} - 1} = l_{p_{1} + a + 1} = k_{p_{2} - 1} =
l_{p_{2}} = k_{p_{2} + a}$. In this form, consider the following
evaluation in $m$: $x_{i} = e_{l_{i}k_{i}} \ \ \forall \ \ i \in
\widehat{q} - \{p_{1},\cdots,p_{1} + a\}, x_{l_{p_{1} + j}k_{p_{1} +
j}} = e_{l_{p_{2} + j}k_{p_{2} + j}} \ \ \forall \ \ j \in
\{0,1,\cdots,a\}$.

Thus $(e_{l_{1}k_{1}}\cdots e_{l_{p_{1}
-1}k_{p_{1}-1}}).(e_{l_{p_{2}}k_{p_{2}}}\cdots e_{l_{p_{2} + a
}k_{p_{2} + a}}).(e_{l_{p_{1} + a + 1}k_{p_{1} + a + 1}}\cdots
e_{l_{q}k_{q}}) \neq 0$.

This is a contradiction, because $m \in T_{G}(R)$. By induction on
$q$, the result follows.
\end{proof}

\begin{lemma}\label{felipe9.55}
If $m$ is a Type 1-monomial identity of $R$, then $m$ follows from
$(4)$.
\end{lemma}
\begin{proof}
It follows from Remark \ref{silva4} and Lemma \ref{felipe9.45}.
\end{proof}

\section{The main result}

The next lemma follows an idea of (\cite{Azevedo1}, Lemma 5),
(\cite{Azevedo2}, Lemma 5), (\cite{Vasilovsky 1}, Lemma 4) and
(\cite{Silva}, Lemma 4.6).

\begin{lemma}\label{felipe5.5}
Let $m(x_{1},\cdots,x_{q})$ and $n(x_{1},\cdots,x_{q})$ be two
monomials such that the matrices $n(A_{1},\cdots,A_{q})$ and
$m(A_{1},\cdots, A_{q})$ have in same position the same non-zero
entry. Then:
\begin{center}
$m(x_{1},x_{2},\cdots,x_{q}) \equiv n(x_{1},x_{2},\cdots,x_{q}) \ \
mod \ \ J_{1}$.
\end{center}
\end{lemma}
\begin{proof}
Let $m = x_{i_{1}}\cdots x_{i_{r}}$. According to Lemma
\ref{felipe4}, $m-n$ is a multi-homogeneous polynomial. Let $m_{1}$
and $n_{1}$ be two multilinear monomials (with the same variables)
such that $h(m_{1}) = h(m)$ and $h(n_{1}) = h(n)$. Note that, it is
enough to prove:
\begin{center}
$m_{1} \equiv n_{1} \ \ mod \ \ J_{1}$.
\end{center}

We may suppose that $m_{1} = x_{1}\cdots x_{r}$. Therefore, there
exists $\sigma \in S_{r}$ such that $n_{1} = x_{\sigma(1)}\cdots
x_{\sigma(r)}$. By this hypothesis, there is a position $(i,j) \in
\widehat{n}\times \widehat{n}$ such that
$e_{1i}m_{1}(A_{1},\cdots,A_{q})e_{j1} =
e_{1i}n_{1}(A_{1},\cdots,A_{q})e_{j1} \neq 0$.

Suppose that the entry of $m_{1}(A_{1},\cdots,A_{q})$, in the
position $(i,j)$ is: $y_{\alpha(x_{1}),1}^{q_{1}}\cdots
y_{\alpha(x_{r}),r}^{q_{r}}$ where $q_{2},\cdots,q_{r} \in
\widehat{n}$ and $q_{1} = i$. Therefore:
\begin{center}
$e_{q_{1}s_{\alpha(x_{1})}^{q_{1}}}\cdots
e_{q_{r}s_{\alpha(x_{r})}^{q_{r}}} =
e_{q_{\sigma(1)}s_{\alpha(x_{\sigma(1)})}^{q_{\sigma(1)}}}\cdots
e_{q_{\sigma(r)}s_{\alpha(x_{\sigma(r)})}^{q_{\sigma(r)}}} =
e_{ij}$.
\end{center}

In this form, there exist matrix units $e_{i_{1}j_{1}} \in
M_{n}(K)_{\alpha(x_{1})},\cdots, e_{i_{r}j_{r}} \in
M_{n}(K)_{\alpha(x_{r})}$ having the following property:
\begin{center}
$e_{i_{1}j_{1}}\cdots e_{i_{r}j_{r}} =
e_{i_{\sigma(1)}j_{\sigma(1)}}\cdots e_{i_{\sigma(r)}j_{\sigma(r)}}
\neq 0$.
\end{center}
So, $i_{1} = i_{\sigma(1)}$, $j_{r} = j_{\sigma(r)}$ and
$\alpha(m_{1}) = \alpha(n_{1}) = g_{i}^{-1}g_{j}$.

In the following steps, we will be use an induction on $r$. If $r =
1$, the proof is obvious.

\begin{description}
\item Step 1: Suppose that $\sigma(1) = 1$. In this case, the monomials $m_{1}$ and $n_{1}$
start with the same variable. Let $m_{2}$ and $n_{2}$ be two
monomials obtained from $m_{1}$ and $n_{1}$ respectively by deleting
the first variable. By Lemma \ref{felipe5},
$m_{2}(A_{1},\cdots,A_{q})$ and $n_{2}(A_{1},\cdots,A_{q})$ have, in
the same position, the same non-zero entry. Hence, by induction
hypothesis, $m_{2} \equiv n_{2}$ modulo $J$. Consequently, $m_{1}
\equiv n_{1}$ modulo $J_{1}$ as required.
\item Step 2: Suppose that $\sigma(1) > 1$. Let $t$ be the least positive integer
such that $\sigma^{-1}(t+1) < \sigma^{-1}(1)\leq \sigma^{-1}(t)$. We
define: $k_{1} := \sigma^{-1}(t+1)$, $k_{2}:= \sigma^{-1}(1)$ and
$k_{3} = \sigma^{-1}(t)$. Note that: $\sigma(k_{1}) = t + 1$,
$\sigma(k_{2}) = 1$, $\sigma(k_{3}) = t$, $i_{\sigma(c+1)} =
j_{\sigma(c)}$ and $i_{c+1} = j_{c}$  for all $c \in \widehat{r-1}$.
It is clear that:

\begin{center}
$n_{1} = x_{\sigma(1)}\cdots x_{\sigma(r)} =
n_{1}^{[1,k_{1}-1]}n_{1}^{[k_{1},k_{2}-1]}n_{1}^{[k_{2}
,k_{3}]}n_{1}^{[k_{3}+1,r]}$.
\end{center}

Likewise:

\begin{center}
$\alpha(n_{1}^{[1,k_{1}-1]}) = g_{i_{\sigma(1)}}^{-1}g_{j_{\sigma(k_{1}-1)}} = g_{i_{1}}^{-1}g_{i_{\sigma(k_{1})}} = g_{i_{1}}^{-1}g_{i_{t+1}}$; \\
$\alpha(n_{1}^{[k_{1},k_{2}-1]}) = g_{i_{\sigma(k_{1})}}^{-1}g_{j_{\sigma(k_{2}-1)}} = g_{i_{t+1}}^{-1}g_{i_{\sigma(k_{2})}} = g_{i_{t+1}}^{-1}g_{i_{1}}$; \\
$\alpha(n_{1}^{[k_{2},k_{3}]}) =
g_{i_{\sigma(k_{2})}}^{-1}g_{j_{\sigma(k_{3})}} =
g_{i_{1}}^{-1}g_{j_{t}} = g_{i_{1}}^{-1}g_{i_{t+1}}$.
\end{center}
Thus, by the identities $(1)$ and $(2)$, it is possible to conclude
that:
\begin{center}
$n_{1} \equiv
n_{1}^{[k_{2},k_{3}]}n_{1}^{[k_{1},k_{2}-1]}n_{1}^{[1,k_{1}-1]}n_{1}^{[k_{3}+1,r]}
\ \ mod \ \ J_{1}$.
\end{center}
Conclusion: $n_{1}$ is a congruent monomial that starts with the
same variable of $m_{1}$. Repeating the arguments of the first case,
we conclude that $m_{1} \equiv n_{1}$ modulo $J_{1}$.
\end{description}
\end{proof}

\begin{lemma}\label{silva2}
Let $G$ be a finite group of order $n$. Then $R$ does not satisfy a
$G$-graded monomial identity.
\end{lemma}
\begin{proof}
According to Remarks \ref{silva4} and \ref{silva5}, it is sufficient
to prove that $M_{n}(K)$ does not satisfy a multilinear monomial
identity $x_{1}.\ldots.x_{l}$.

It is clear that $Supp_{G}(M_{n}(K)) = G$. So, it is enough to prove
that $M_{n}(K)$ does not satisfy a Type-1 multilinear monomial
identity. If $l = 1$, the proof is obvious. The proof is made by
induction on $l$. According to the hypothesis of induction, there
exist matrix units $e_{i_{2}j_{2}} \in (M_{n}(K))_{\alpha(x_{2})},
\cdots,e_{i_{l}j_{l}} \in (M_{n}(K))_{\alpha(x_{l})}$ such that
$e_{i_{2}j_{2}}\cdots e_{i_{l}j_{l}} = e_{i_{2}j_{l}}$. Notice that
there exists $g_{k} \in \{g_{1},\ldots,g_{n}\}$ such that
$g_{k}^{-1}g_{i_{2}} = \alpha(x_{1})$. So, $e_{ki_{2}}\cdots
e_{i_{l}j_{l}} \neq 0$. The proof of Lemma \ref{silva2} is complete.
\end{proof}

\begin{lemma}\label{felipe6}
If $R$ does not satisfy a monomial identity, then $T_{G}(R) =
J_{1}$.
\end{lemma}
\begin{proof}
Suppose for contradiction there exists $f(x_{1},\cdots,x_{t}) =
\sum_{i=1}^{l} \lambda_{i}m_{i} \in T_{G}(R) - J_{1}$, where for all
$i \in \widehat{l}$: $\lambda_{i} \in K-\{0\}$, $m_{i}$ is a
monomial. According to Lemma \ref{felipe3}, we may suppose that $f$
is a multi-homogeneous polynomial. Moreover, it can be supposed that
$l$ is the least positive integer with the following set:
\begin{center}
$B = \{q \in \mathbb{N}| \sum_{i=1}^{q}
\gamma_{i}n_{i}(x_{1},\cdots,x_{t}) \in T_{G}(R) - J_{1}; \gamma_{i}
\in K - \{0\} \ \mbox{for all} \ i \in \widehat{q}\}$.
\end{center}

It is clear that $m_{i}(A_{1},\cdots,A_{t}) \neq 0$, for $i =
1,\cdots,l$, because $R$ does not satisfy a monomial identity.
Furthermore, there exists $k \in \{2,\cdots,l\}$ such that:
\begin{center}
$m_{1}(A_{1},\cdots,A_{t})$ and $m_{k}(A_{1},\cdots,A_{t})$
\end{center}
have, in the same position the same non-zero entry. Thus by Lemma
\ref{felipe5.5}, it follows that $m_{1} \equiv m_{k}$ modulo
$J_{1}$. Consequently, $h = f + \lambda_{k}(m_{1} - m_{k})  \in
T_{G}(R) - J_{1}$. The contraction, in addition the number of
non-zero summands in $h$ is less than $l$.
\end{proof}

\begin{corollary}
Let $K$ be an infinite integral domain. The $\mathbb{Z}_{n}$-graded
polynomial identities of $M_{n}(K)$ follow from:
\begin{center}
$\bullet \ x_{1}x_{2} - x_{2}x_{1} \ \ \mbox{when} \ \ \alpha(x_{1})
= \alpha(x_{2}) = \overline{0} \
\ (1)$; \\

$\bullet \ x_{1}x_{2}x_{3} - x_{3}x_{2}x_{1} \ \ \mbox{when} \ \
\alpha(x_{1}) = \alpha(x_{3}) = - (\alpha(x_{2})) \neq \overline{0}
\ \ (2)$.
\end{center}
\end{corollary}
\begin{proof}
It follows from Lemmas \ref{silva2} and \ref{felipe6}.
\end{proof}

Following word for word the work of Vasilovsky in \cite{Vasilovsky
2} (see Lemma 1, Lemma 3 and Corollary 4), we have the following
lemma:

\begin{lemma}
Let $K$ be an infinite integral domain and $m = x_{1}\cdots x_{l}$
be a multilinear Type 1-monomial. Then $m \notin
T_{\mathbb{Z}}(M_{n}(K))$.
\end{lemma}

\begin{corollary}\label{felipe7}
Let $K$ be an infinite integral domain and $m$ be a Type 1-monomial.
Then $m \notin T_{\mathbb{Z}}(M_{n}(K))$.
\end{corollary}

Following word for word the proof of Lemma \ref{felipe6}, we have
the following lemma:

\begin{lemma}\label{silva3}
If $R$ does not satisfy a Type 1-monomial identity, then $T_{G}(R) =
J$.
\end{lemma}

\begin{corollary}
Let $K$ be an infinite integral domain. The $\mathbb{Z}$-graded
polynomial identities of $M_{n}(K)$ follow from:
\begin{center}
$\bullet \ x_{1}x_{2} - x_{2}x_{1} \ \ \mbox{when} \ \ \alpha(x_{1})
= \alpha(x_{2}) = 0 \
\ (1)$; \\

$\bullet \ x_{1}x_{2}x_{3} - x_{3}x_{2}x_{1} \ \ \mbox{when} \ \
\alpha(x_{1}) = \alpha(x_{3})
= - (\alpha(x_{2})) \neq 0  \ \ (2)$; \\

$\bullet \ x_{1} \ \ \mbox{when} \ \ |\alpha(x_{1})| \geq n \ \
(3)$.
\end{center}
\end{corollary}
\begin{proof}
It follows from Corollary \ref{felipe7} and Lemma \ref{silva3}.
\end{proof}

Now, we present the main result of this paper.

\begin{theorem}
Let $G$ be an arbitrary group. Then $T_{G}(R) = U$. If $K$ is an
infinite integral domain, then $T_{G}(M_{n}(K)) = U$.
\end{theorem}
\begin{proof}
According to Lemmas \ref{luis4} and \ref{felipe9.55}, if $m$ is a
monomial identity of $R$, then $m$ is consequence of $(3)$ or $(4)$.
Thus, it is sufficient to imitate the proof of Lemma \ref{felipe6}
and replace $J_{1}$ with $U$, early in the proof.
\end{proof}

\section{Matrix-units graded identities of $M_{n}(K)$ over an infinite integral domain}

The algebra $M_{n}(K)$ has a natural grading by $MU_{n}$, the
semigroup of matrix units of class $n$.

\begin{definition}
Let $MU_{n} = \{(i,j) \in \widehat{n}\times \widehat{n} \}\cup
\{0\}$ denote the semigroup of matrix units of class $n$ whose
multiplication is defined as follows:
\begin{description}
\item $0.(i,j) = (i,j).0 = 0$;
\item $(i,j)(k,l) = (i,l)$ when $j = k$;
\item $(i,j)(k,l) = 0$ when $j\neq k$.
\end{description}
\end{definition}

Let $x_{0}$ and $x_{ij},y_{ij},z_{ij}$ denote the free variables
whose $MU_{n}$-degree are $0$ and $(i,j)$, respectively. The
following theorem addresses this issue:

\begin{theorem}(\cite{Drensky2}, Theorem 4.9)
Let $K$ be a field of characteristic zero. Then, the $MU_{n}$-graded
identities of $M_{n}(K)$ follow from:
\begin{center}
$x_{ii}y_{ii} - y_{ii}x_{ii} \ \ \mbox{when} \ \ i \in \widehat{n} \ \ (5)$; \\
$x_{ij}y_{ji}z_{ij} - z_{ij}y_{ji}x_{ij} \ \ \mbox{when} \ \ 1\leq i,j \leq n \ \ i\neq j \ \ (6)$; \\
$x_{0} \ \ (7)$.
\end{center}
\end{theorem}

Here, we extend this result for infinite integral domains. Let
$J_{MU_{n}}$ be the $T_{MU_{n}}$-ideal generated by $(5)$, $(6)$ and
$(7)$.

Let $h \in MU_{n} - \{0\}$. Let $W_{h} =
\{w_{h}^{(1)},\cdots,w_{h}^{(n)},\cdots\}$ denote a countable set of
commuting variables. Let $W = \bigcup_{h \in MU_{n} - \{0\}} W_{h}$
and let $\Omega_{1} = K[W]$ be the polynomial ring in commuting
variables of the set $W$.

\begin{definition}
A generic matrix of $M_{n}(\Omega_{1})$ of $MU_{n}$-degree $(i,j)$
is a homogeneous element of the following type:
\begin{center}
$A_{(i,j)}^{(k)} := w_{(i,j)}^{(k)}e_{ij}$ where $1 \leq i,j \leq n$
and $k \in \{1,2,\cdots\}$.
\end{center}
The $MU_{n}$-graded subalgebra generated by the generic matrices of
$M_{n}(\Omega_{1})$ is called the algebra of generic matrices which
we denote by $R_{1}$.
\end{definition}

Notice that if $m$ is a $MU_{n}$-graded monomial identity of $R$,
then $m$ follows from $(7)$. The main steps of the proof of Lemmas
\ref{felipe5.5} and \ref{felipe6} hold also for this grading and we
obtain the following result.

\begin{theorem}
The $MU_{n}$-graded identities of $R_{1}$ follow from:
\begin{center}
$\bullet \ x_{ii}y_{jj} - y_{jj}x_{ii} \ \ \mbox{when} \ \ i \in \widehat{n} \ \ (5);$ \\
$\bullet \ x_{ij}y_{ji}z_{ij} - z_{ij}y_{ji}x_{ij} \ \ \mbox{when} \ \ 1 \leq i,j \leq n, \ \ i\neq j \ \ (6);$ \\
$\bullet \ x_{0} \ \ (7).$
\end{center}
If $K$ is an infinite integral domain, then $T_{MU_{n}}(M_{n}(K)) =
J_{MU_{n}}$.
\end{theorem}

In the rest of this paper we will only consider matrices over an
infinite integral domain.

\section{$\mathbb{Z}_{p}$-graded central polynomial of $M_{p}(K)$}

Now, we describe the $\mathbb{Z}_{p}$-graded central polynomials of
$M_{p}(K)$ when $p$ is a prime number.

Let $\pi: \mathbb{Z} \rightarrow \mathbb{Z}_{n}$ denote the
canonical projection and let $j \in \mathbb{Z}_{n}$. The following
convention will be described in this section:
\begin{center}
$y_{h,i}^{j} := y_{h,i}^{k}$ when $k = \pi^{-1}(j)\cap \widehat{n}$.
\end{center}

\begin{definition}(\cite{Junior}, Preliminaries)
A sequence $(\gamma_{1},\cdots,\gamma_{n})$ of elements of
$\mathbb{Z}_{n}$ is called a complete sequence when the following
conditions are satisfied:
\begin{description}
\item $\gamma_{1} + \cdots + \gamma_{n} = 0$;
\item $\{\gamma_{1},\gamma_{1} + \gamma_{2},\cdots,\gamma_{1} + \cdots + \gamma_{n}\} =
\mathbb{Z}_{n}$.
\end{description}
\end{definition}

The next lemma is the immediate consequence of the complete sequence
definition.

\begin{lemma}\label{felipe14.4}
A sequence $(\gamma_{1},\cdots,\gamma_{n})$ of elements of
$\mathbb{Z}_{n}$ is a complete sequence of $\mathbb{Z}_{n}$ if and
only if there exist matrix units $e_{i_{1}j_{1}} \in
M_{n}(K)_{\gamma_{1}},\cdots, e_{i_{n}j_{n}} \in
M_{n}(K)_{\gamma_{n}}$ such that $i_{l + 1} = j_{l}$ for all $(l +
1)\in \widehat{n}$. Moreover, $\widehat{n} = \{i_{1},\cdots,i_{n}\}$
and $i_{1} = j_{n}$.
\end{lemma}

Following word for word the proof of Brandão Júnior in \cite{Junior}
(see Lemma 1 and Proposition 1), we have the following lemma:

\begin{lemma}\label{felipe14.5}
The $\mathbb{Z}_{n}$-graded multilinear polynomial:
\begin{center}
$\sum_{\sigma \in H_{n}} x_{\sigma(1)}\cdots x_{\sigma(n)}$,
\end{center}
where $({\alpha(x_{1})},\cdots,\alpha(x_{n}))$ is a complete
sequence of $\mathbb{Z}_{n}$, is a $\mathbb{Z}_{n}$-graded central
polynomial of $M_{n}(K)$. Furthermore, it is not a
$\mathbb{Z}_{n}$-graded polynomial identity of $M_{n}(K)$.
\end{lemma}

\begin{lemma}\label{felipe14.6}
Let $m = x_{i_{1}}\cdots x_{i_{q}}$ be a monomial such that
$\alpha(m) = 0$ and $\alpha(x_{i_{j}}) = h_{i_{j}} \ \mbox{for all}
\ j \in \widehat{q}$. Let $y_{h_{i_{1}},i_{1}}^{s_{i_{1}}^{k}}\cdots
y_{h_{i_{q}},i_{q}}^{s_{i_{q}}^{k}}$ be an entry of
$A_{i_{1}}.\cdots.A_{i_{q}}$. If there exists a subsequence
$(s_{i_{v_{1}}}^{k},\cdots,s_{i_{v_{n}}}^{k})$ of
$(s_{i_{1}}^{k},\cdots,s_{i_{q}}^{k})$ such that
$\{s_{i_{v_{1}}}^{k},\cdots,s_{i_{v_{n}}}^{k}\} = \widehat{n}$ and
$s_{i_{v_{1}}}^{k} = s_{i_{1}}^{k}$, then there exist monomials
$m_{1},\cdots,m_{n}$ such that
$(\alpha(m_{1}),\cdots,\alpha(m_{n}))$ is a complete sequence of
$\mathbb{Z}_{n}$ and $m = m_{1}m_{2}\cdots m_{n}$.
\end{lemma}
\begin{proof}
If $n = 1$, the proof is obvious. From now on, $n > 1$. First, we
assume that $m$ is a multilinear monomial.

In fact, there are matrix units $e_{j_{1}l_{1}} \in
M_{n}(K)_{\alpha(x_{1})},\cdots,e_{j_{q}l_{q}} \in
M_{n}(K)_{\alpha(x_{q})}$ such that $e_{j_{1}l_{1}}\cdots
e_{j_{q}l_{q}} = e_{j_{1}l_{q}}$, where $j_{t} = s_{i_{t}}^{k}$ for
all $t \in \widehat{q}$. Likewise, $j_{1} = l_{q}$ because $m \in
K\langle X \rangle_{0}$. We may suppose without any loss of
generality that $j_{t} = s_{i_{v_{t}}}^{k}$ for $t = 2,\cdots,n$.

Let: $m_{i} = x_{i}$, $i = 1,\cdots,n-1$; $m_{n} = x_{n}\cdots
x_{q}$.

Notice that $e_{j_{1}l_{1}}e_{j_{2}l_{2}}\cdots e_{j_{n}l_{q}} =
e_{j_{1}l_{q}}$. So $\{j_{1},\cdots,j_{n}\} = \widehat{n}$, $l_{t} =
j_{t+1}$ for all $t \in \widehat{n-1}$ and $j_{1} = l_{q}$.
Consequently, by Lemma \ref{felipe14.4}, it follows that
$m_{1},\cdots,m_{n}$ satisfy the thesis of this lemma.

Now, we assume that $m$ is an arbitrary monomial. We would choose a
multilinear monomial $\overline{m}$ such that $h(m) =
h(\overline{m})$. There exist monomials
$\overline{m_{1}},\cdots,\overline{m_{n}}$ such that
$(\alpha(\overline{m_{1}}),\cdots,\alpha(\overline{m_{n}}))$ is a
complete sequence of $\mathbb{Z}_{n}$ and $\overline{m} =
\overline{m_{1}}\overline{m_{2}}\cdots \overline{m_{n}}$. Thus,
there must exist monomials $m_{1},\cdots,m_{n}$ such that $m =
m_{1}.\cdots.m_{n}$ and $h(m_{1}) =
h(\overline{m_{1}}),\cdots,h(m_{n}) = h(\overline{m_{n}})$. The
proof is complete.
\end{proof}

The proof of the following lemma is left as an exercise.

\begin{lemma}\label{felipe16.9}
Let $A = \{a_{1},\cdots,a_{l}\} \varsubsetneq \mathbb{Z}_{p}$ be a
set. Then:
\begin{center}
$\{a_{1} + i,\ldots, a_{l} + i \} \neq \{a_{1} + j,\ldots, a_{l} + j
\}$
\end{center}
for any $i,j \in \mathbb{Z}_{p}$ distinct.
\end{lemma}

The next lemma is well known.

\begin{lemma}\label{felipe15.49}
Let $z_{1},z_{2} \in K\langle X \rangle_{1}$. Then the monomials
$z_{1}^{2}$ and $z_{1}^{2}z_{2}^{2}$ are $\mathbb{Z}_{2}$-graded
central monomials of $M_{2}(K)$.
\end{lemma}

\begin{lemma}\label{felipe16.5}
Let $p > 2$. Let $x_{1},x_{2}$ be variables such that $\alpha(x_{1})
= \alpha(x_{2})\neq 0$. Then $(x_{1}x_{2})^{p} \equiv
x_{2}^{p}x_{1}^{p} \ \ mod \ \ T_{\mathbb{Z}_{p}}(M_{p}(K))$.
\end{lemma}
\begin{proof}
Let $A_{1} \in R_{\alpha(x_{1})}$ and $A_{2} \in R_{\alpha(x_{2})}$
be two generic matrices. By Lemma \ref{felipe0}, it is obvious that
all positions in the diagonal of $(A_{1}A_{2})^{p}$ (respectively
$A_{2}^{p}A_{1}^{p}$) have non-zero entries. According to Lemma
\ref{felipe5.5}, it is sufficient to prove that
$e_{11}(A_{1}A_{2})^{p} = e_{11}(A_{2}^{p}A_{1}^{p})$.

In fact,
\begin{center}
$e_{11}(A_{1}A_{2})^{p} =
e_{11}(\sum_{i=1}^{p}y_{\alpha(x_{1}),1}^{i}y_{\alpha(x_{1}),2}^{i
+\alpha(x_{1})}e_{i\pi^{-1}(\overline{i + 2\alpha(x_{1})})\cap
\widehat{p}})^{p} =
(\prod_{i=1}^{p}(y_{\alpha(x_{1}),1}^{i}y_{\alpha(x_{1}),2}^{i +
\alpha(x_{1})})e_{11}) = (\prod_{i=1}^{p} y_{\alpha(x_{1}),2}^{i +
\alpha(x_{1})})(\prod_{i=1}^{p} y_{\alpha(x_{1}),1}^{i})e_{11} =
(\prod_{i=1}^{p} y_{\alpha(x_{1}),2}^{i})(\prod_{i=1}^{p}
y_{\alpha(x_{1}),1}^{i})e_{11} = (A_{2})^{p}e_{11}(A_{1})^{p}e_{11}
= A_{2}^{p}A_{1}^{p}e_{11}$.
\end{center}
So, $(x_{1}x_{2})^{p} \equiv x_{2}^{p}x_{1}^{p} \ \ mod \ \
T_{\mathbb{Z}_{p}}(M_{p}(K))$ as required.
\end{proof}

\begin{lemma}\label{felipe15.5}(\cite{Junior}, Lemma 8)
Let $p > 2$ and let $l \in \widehat{p - 1}$. Let $m =
x_{1}^{p}\cdots x_{l}^{p}$ be a $\mathbb{Z}_{p}$-graded monomial
such that $\alpha(x_{i}) \neq \alpha(x_{j})$ for $i\neq j$. Then $m
\in C_{\mathbb{Z}_{p}}(M_{p}(K))$.
\end{lemma}
\begin{proof}
First, we assume that $l = 1$.

According to Lemma \ref{felipe0}
\begin{center}
$(A_{1}^{\alpha(x_{1})})^{p} = \sum_{i = 1}^{p}
y_{\alpha(x_{1}),1}^{i}y_{\alpha(x_{1}),1}^{i + \alpha(x_{1})}\dots
y_{\alpha(x_{1}),1}^{i + (p - 1)\alpha(x_{1})}e_{ii}$.
\end{center}
So $y_{\alpha(x_{1}),1}^{i}y_{\alpha(x_{1}),1}^{i +
\alpha(x_{1})}\dots y_{\alpha(x_{1}),1}^{i + {(p - 1)\alpha(x_{1})}}
= y_{\alpha(x_{1}),1}^{j}y_{\alpha(x_{1}),1}^{j +
\alpha(x_{1})}\dots y_{\alpha(x_{1}),1}^{j + (p - 1)\alpha(x_{1})}$
for all $i,j \in \widehat{p}$ because $\langle \alpha(x_{1}) \rangle
= \mathbb{Z}_{p}$. Consequently $x_{1}^{p} \in
C_{\mathbb{Z}_{p}}(M_{p}(K))$. Bearing in mind that
$C_{\mathbb{Z}_{p}}(M_{p}(K))$ is a subalgebra, the result follows.
\end{proof}

\begin{definition}\label{felipe16}
Let $V_{1}$ denote the $T_{\mathbb{Z}_{p}}$-graded space generated
by the monomials reported in the hypothesis of Lemma
\ref{felipe15.49} or the monomials that satisfy the hypothesis of
Lemma \ref{felipe15.5}.
\end{definition}

The proofs of the two following lemmas (Lemmas \ref{luis2} and
\ref{luis3}) are left as an exercise.

\begin{lemma}\label{luis2}
Let $z_{1},\cdots,z_{n} \in K\langle X \rangle_{1}$. Let $m =
z_{1}^{2.k_{1}}\cdots z_{n}^{2.k_{n}}$, where $k_{1},\cdots,k_{n}
\in \mathbb{N} - \{0\}$. Then there exists $m_{1} \in \langle
V_{1}\rangle^{T_{\mathbb{Z}_{2}}}$ such that $m - m_{1} \equiv 0$
modulo $T_{\mathbb{Z}_{2}}(M_{2}(K))$.
\end{lemma}

\begin{lemma}\label{luis3}
Let $p > 2$ and let $x_{1},\cdots,x_{n} \in K\langle X \rangle_{i},
i \neq 0$. Let $m = x_{1}^{p.k_{1}}\cdots x_{n}^{p.k_{n}}$ be a
monomial, where $k_{1},\cdots,k_{n} \in \mathbb{N} - \{0\}$. Then
there exists $m_{1} \in \langle V_{1}\rangle^{T_{\mathbb{Z}_{p}}}$
such that $m - m_{1} \equiv 0$ modulo
$T_{\mathbb{Z}_{p}}(M_{p}(K))$.
\end{lemma}

\begin{lemma}\label{felipe22}
Let $m = x_{i_{1}} \cdots x_{i_{r}} \in
C_{\mathbb{Z}_{p}}(M_{p}(K))$. Then there exists $m_{1} \in \langle
V_{1}\rangle^{T_{\mathbb{Z}_{p}}}$ such that $m - m_{1} \equiv 0$
modulo $T_{\mathbb{Z}_{p}}(M_{p}(K))$.
\end{lemma}
\begin{proof}
Notice that $m \notin T_{\mathbb{Z}_{p}}(M_{p}(K))$. Moreover, at
least one variable of $m$ has $\mathbb{Z}_{p}$-degree different than
$0$.

By hypothesis, $m$ is a $\mathbb{Z}_{p}$-graded central monomial of
$M_{p}(K)$. Consequently, using the Lemma \ref{felipe0}, it follows
that:
\begin{center}
$y_{\alpha(x_{i_{1}}),i_{1}}^{i}y_{\alpha(x_{i_{2}}),i_{2}}^{i +
\alpha(x_{i_{1}})}\cdots y_{\alpha(x_{i_{r}}),i_{r}}^{i +
\alpha(x_{i_{1}} + \cdots x_{i_{r-1}})} =
y_{\alpha(x_{i_{1}}),i_{1}}^{j}y_{\alpha(x_{i_{2}}),i_{2}}^{j +
\alpha(x_{i_{1}})}\cdots y_{\alpha(x_{i_{r}}),i_{r}}^{j +
\alpha(x_{i_{1}} + \cdots x_{i_{r-1}})}$
\end{center}
for any $i,j \in \widehat{p}$.

Let $x_{l_{1}},\cdots,x_{l_{q}}$ be all the different variables of
the monomial $m$. $k_{i}$ denotes the ordinary degree of $m$ with
respect to variable $x_{l_{i}}$. Note that, for each $i \in
\widehat{q}$, $k_{i}$ is a multiple of $p$.

Case 1: $\alpha(x_{l_{i}}) \neq 0$ for all $i \in \widehat{q}$. Let
$x_{l_{1}}^{k_{1}}\cdots x_{l_{q}}^{k_{q}}$ be a monomial and let
$A_{l_{1}} \in R_{\alpha(x_{l_{1}})},\cdots,A_{l_{q}} \in
R_{\alpha(x_{l_{q}})}$ be generic matrices. Evidently, the matrices
$A_{l_{1}}^{k_{1}}\cdots A_{l_{q}}^{k_{q}}$ and
$m(A_{l_{1}},\cdots,A_{l_{q}})$ have in position $(1,1)$, the same
non-zero entry. Therefore, by Lemma \ref{felipe5.5}, $m \equiv
x_{l_{1}}^{k_{1}}\cdots x_{l_{q}}^{k_{q}} \ \ mod \ \ J_{1}$.
Applying Lemma \ref{luis2} or Lemma \ref{luis3}, we are done.

Case 2: there exists $i \in \widehat{q}$ such that
$\alpha(x_{l_{i}}) = 0$. Suppose that all variables of $G$-degree
$0$ are $\{x_{l_{1}},\cdots,x_{l_{s}}\}$. Choose $m_{1}
=(x_{l_{s+1}}x_{l_{1}}^{k_{l_{1}}}\cdots
x_{l_{s}}^{k_{s}})^{p}x_{l_{s+1}}^{k_{s+1} -
p}x_{l_{s+2}}^{k_{s+2}}\cdots x_{l_{q}}^{k_{l_{q}}}$. Evidently,
$m_{1}(A_{l_{1}},\cdots,A_{l_{q}})$ and
$m(A_{i_{1}},\cdots,A_{i_{r}})$ have in position $(1,1)$ the same
non-zero entry. Applying the ideas of previous case, the result
follows.
\end{proof}

\begin{lemma}\label{felipe17}
Let $m = x_{i_{1}}\cdots x_{i_{q}} \in K\langle X \rangle_{0} -
(C_{\mathbb{Z}_{p}}(M_{p}(K)))\cap(K\langle X \rangle_{0})$. All
entries in the diagonal of $A_{i_{1}}.\cdots.A_{i_{q}}$ are non-zero
and pairwise distinct.
\end{lemma}
\begin{proof}
(Sketches) According to Lemma \ref{felipe0}, all entries in the
diagonal are non-zero. If $p = 2$, the analysis is obvious.

Henceforth, suppose that $p > 2$. By hypothesis, $x_{i_{1}}\cdots
x_{i_{q}} \in K\langle X \rangle_{0} -
C_{\mathbb{Z}_{p}}(M_{p}(K))\cap K\langle X \rangle_{0}$. Thus, the
following condition is satisfied: there exist $j_{1} < \cdots <
j_{l} \in \widehat{q}$, where $x_{i_{j_{1}}} = \cdots =
x_{i_{j_{l}}}$ and $x_{i_{t}} \neq x_{i_{j_{1}}} \ \ \mbox{for all}
\ t \ \in \ \widehat{q} - \{j_{1},\dots,j_{l}\}$ such that:

\begin{center}
$y_{\alpha(x_{i_{j_{1}}}),i_{j_{1}}}^{k_{1} +
\alpha(x_{i_{1}})+\cdots + \alpha(x_{i_{j_{1} - 1}})} \cdots
y_{\alpha(x_{i_{j_{l}}}),i_{j_{l}}}^{k_{1} +
\alpha(x_{i_{1}})+\cdots + \alpha(x_{i_{j_{l} - 1}})} \neq \newline
y_{\alpha(x_{i_{j_{1}}}),i_{j_{1}}}^{k_{2} +
\alpha(x_{i_{1}})+\cdots + \alpha(x_{i_{j_{1} - 1}})} \cdots
y_{\alpha(x_{i_{j_{l}}}),i_{j_{l}}}^{k_{2} +
\alpha(x_{i_{1}})+\cdots + \alpha(x_{i_{j_{l} - 1}})}.$
\end{center}

By Lemma \ref{felipe16.9} and some calculations, we may conclude
that:

\begin{center}
$y_{\alpha(x_{i_{1}}),i_{1}}^{q_{1} + \alpha(x_{i_{1}})+ \cdots +
\alpha(x_{i_{j_{1} - 1}})} \cdots
y_{\alpha(x_{i_{j_{l}}}),i_{j_{l}}}^{q_{1} +
\alpha(x_{i_{1}})+\cdots + \alpha(x_{i_{j_{l} - 1}})} \neq
y_{\alpha(x_{i_{1}}),i_{1}}^{q_{2} + \alpha(x_{i_{1}})+ \cdots +
\alpha(x_{i_{j_{1} - 1}})} \cdots
y_{\alpha(x_{i_{j_{l}}}),i_{j_{l}}}^{q_{2} +
\alpha(x_{i_{1}})+\cdots + \alpha(x_{i_{j_{l} - 1}})}$ for all
$q_{1} \neq q_{2} \in \mathbb{Z}_{p}$.
\end{center}

The proof is complete.
\end{proof}

In what follows, we use substantially the proof of Theorem 6 used by
Brandão Júnior \cite{Junior}.

\begin{theorem}\label{felipe23}
The $\mathbb{Z}_{p}$-graded central polynomials of $M_{p}(K)$ follow
from:
\begin{center}
$z_{1}(x_{1}x_{2} - x_{2}x_{1})z_{2} \ \ \mbox{when} \ \
\alpha(x_{1}) = \alpha(x_{2})
= \overline{0} \ \ (8);$ \\
$z_{1}(x_{1}x_{2}x_{3} - x_{3}x_{2}x_{1})z_{2} \ \ \mbox{when} \ \ \alpha(x_{1}) = - \alpha(x_{2}) = \alpha(x_{3}) \neq \overline{0} \ \ (9);$ \\

The monomials cited in the definition \ref{felipe16} \ \ $(10)$; \\

$\sum_{\sigma \in H_{p}} x_{\sigma(1)}\cdots x_{\sigma(p)}$, where
$(\alpha(x_{1}),\cdots,\alpha(x_{p}))$ is a complete sequence of
$\mathbb{Z}_{p} \ \ (11)$.
\end{center}
The monomials $z_{1},z_{2} \in \bigcup_{g \in \mathbb{Z}_{p}}
X_{g}$.
\end{theorem}
\begin{proof}
Let $W$ be the $T_{\mathbb{Z}_{p}}$-space generated  by
$(8),(9),(10)$ and $(11)$. We prove that
$C_{\mathbb{Z}_{p}}(M_{p}(K)) \subset W$. Let $f(x_{1},\cdots,x_{q})
= \sum_{i=1}^{l}\lambda_{i}m_{i} \in C_{\mathbb{Z}_{p}}(M_{p}(K)) -
T_{\mathbb{Z}_{p}}(M_{p}(K))$. By Lemma \ref{felipe12} and Corollary
\ref{felipe11}, we may assume that $f$ is a multi-homogeneous
polynomial. We may suppose that $\alpha(m_{1}) = \cdots =
\alpha(m_{l}) = 0,$ $m_{i} - m_{j}$ is not an element of
$T_{\mathbb{Z}_{p}}(M_{p}(K))$ and each $m_{i} \notin
C_{\mathbb{Z}_{p}}(M_{p}(K))$ (Lemma \ref{felipe22}).

Let $A_{1} \in R_{\alpha(x_{1})},\cdots,A_{q}\in R_{\alpha(x_{q})}$
be generic matrices. So $f(A_{1},\cdots,A_{q}) =
diag(F_{1},\cdots,F_{p})$, where $F_{1} = \cdots = F_{p} \neq 0$. By
Lemma \ref{felipe17}, for each $j \in \widehat{p}$, all positions in
the diagonal of the matrix $m_{j}(A_{1},\cdots,A_{q})$ have non-zero
entries. Furthermore, these entries are pairwise distinct.

Reordering the indices, if necessary, there exist $1 \leq i_{1} <
\cdots < i_{p} \leq l$ such that $\lambda_{i_{1}} = \cdots =
\lambda_{i_{p}} \neq 0$ and
$e_{11}m_{i_{1}}(A_{1},\cdots,A_{q})e_{11} =
e_{1l_{2}}m_{i_{l_{2}}}(A_{1},\cdots,A_{q})e_{l_{2}1}$ for all
$l_{2} \in \widehat{p} - \{1\}$. Assume that $m_{i_{1}} =
x_{j_{1}}\cdots x_{j_{s}}$ and the entry in position $(1,1)$ of
$m_{i_{1}}(A_{1},\cdots,A_{q})$ is
$y_{\alpha(x_{j_{1}}),j_{1}}^{a_{1}}\cdots
y_{\alpha(x_{j_{s}}),j_{s}}^{a_{s}}$. Notice that the multi-set
$\{a_{1},\cdots,a_{s}\}$ contains $\widehat{p}$.

According to Lemma \ref{felipe14.6}, there are monomials
$r_{1},\cdots,r_{p}$ such that $m_{i_{1}} = r_{1}\cdots r_{p}$ where
$(\alpha(r_{1}),\cdots,\alpha(r_{p}))$ is a complete sequence of
$\mathbb{Z}_{p}$. For each $j \in \widehat{p}$, there is a unique
permutation $\sigma \in H_{p}$ such that the matrices:
\begin{center}
$m_{i_{j}}(A_{1},\cdots,A_{q}) \ \ \mbox{and} \ \
r_{\sigma(1)}\cdots r_{\sigma(p)}(A_{1},\cdots,A_{q})$
\end{center}
have, in the position $(1,1)$, the same non-zero entry. By Lemma
\ref{felipe5.5}:
\begin{center}
$m_{i_{j}} \equiv r_{\sigma(1)}\cdots r_{\sigma(p)} \ \ mod \ \
T_{\mathbb{Z}_{p}}(M_{p}(K))$.
\end{center}
According to Lemma \ref{felipe14.5}, it is clear that:
\begin{center}
$g(x_{1},\cdots,x_{r}) = \lambda_{i_{1}}(\sum_{\sigma \in
H_{p}}r_{\sigma(1)}\cdots r_{\sigma(p)}) \in W$.
\end{center}

Then, $f - g \equiv f - \lambda_{i_{1}}(m_{i_{1}} + m_{i_{2}} +
\cdots + m_{i_{p}})$ modulo $T_{\mathbb{Z}_{p}}(M_{p}(K))$. If $l -
p = 0$, it follows that $f \in W$. If $l - p \geq p$ or $1 \leq l -
p \leq p-1$, the same argument can be repeated. From an inductive
argument on $l$, the result follows.
\end{proof}

\section{$\mathbb{Z}$-graded central polynomials of $M_{n}(K)$}

In this section, we use the same technique as in previous section to
detail a script of the proof. The first Lemma is similar to Lemma
\ref{felipe17}.

\begin{lemma}\label{felipe20}
Let $x_{i_{1}}.\cdots .x_{i_{m}} \in K\langle X\rangle_{0}$.

If
$A_{i_{1}}^{\alpha(x_{i_{1}})}.\cdots.A_{i_{m}}^{\alpha(x_{i_{m}})}$
is a non-zero matrix, then all non-zero entries of that matrix are
pairwise distinct.
\end{lemma}
\begin{proof}
(Sketches) If only one position in $A_{i_{1}}.\cdots.A_{i_{m}}$ has
a non-zero entry, the proof is obvious. Suppose that there exist, at
least, two positions $(k_{1},k_{1}),(k_{2},k_{2}) \in
\widehat{n}\times \widehat{n}$ such that:
\begin{center}
$e_{1k_{1}}(A_{i_{1}}\cdots A_{i_{m}})e_{k_{1}1},
e_{1k_{2}}(A_{i_{1}}\cdots A_{i_{m}})e_{k_{2}1} \neq 0$.
\end{center}
Our aim is to prove that:
\begin{center}
$e_{1k_{1}}(A_{i_{1}}\cdots A_{i_{m}})e_{k_{1}1} \neq
e_{1k_{2}}(A_{i_{1}}\cdots A_{i_{m}})e_{k_{2}1}$.
\end{center}
Suppose by contradiction that:
\begin{center}
$e_{1k_{1}}(A_{i_{1}}\cdots A_{i_{m}})e_{k_{1}1} =
e_{1k_{2}}(A_{i_{1}}\cdots A_{i_{m}})e_{k_{2}1}$.
\end{center}
By contradiction, the result follows.
\end{proof}

The main steps of the proof of the Lemmas \ref{felipe14.5}, Lemma
\ref{felipe14.6} and Theorem \ref{felipe23} hold also for this
grading and we obtain the following result.

\begin{theorem}
Let $K$ be an infinite integral domain. Then the $\mathbb{Z}$-graded
central polynomials of $M_{n}(K)$ follow from:
\begin{center}
$z_{1}(x_{1}x_{2} - x_{2}x_{1})z_{2} \ \ \mbox{when} \ \
\alpha(x_{1}) =
\alpha(x_{2}) = 0 \ \ (12)$; \\
$z_{1}(x_{1}x_{2}x_{3} - x_{3}x_{2}x_{1})z_{2} \ \ \mbox{when} \ \
\alpha(x_{1}) =
-\alpha(x_{2}) = \alpha(x_{3}) \neq 0 \ \ (13)$; \\
$z_{1}x_{1}z_{2} \ \ \mbox{when} \ \ |\alpha(x_{1})| \geq n \ \ (14)$; \\
$\sum_{\sigma \in H_{n}} x_{\sigma(1)}\cdots x_{\sigma(n)}$, where
$(\overline{\alpha(x_{1})},\cdots,\overline{\alpha(x_{n})})$ is a
complete sequence of $\mathbb{Z}_{n}$ and $|\alpha(x_{i})| < n \ \
(15)$.
\end{center}
The monomials $z_{1},z_{2} \in \bigcup_{i \in \mathbb{Z}} X_{i}$.
\end{theorem}

\section{Acknowledgments}

My sincerely thanks go to Alexei N. Krasilnikov, for proposing the
problem, and my entire doctorate board (Alexei Krasilnikov, Plamen
Koshlukov, Viviane Ribeiro Tomaz da Silva, Victor Petrogradsky and
José Antônio Oliveira Freitas) for their useful advice, comments,
remarks and suggestions.

The author would like to thank the reviewers of Rendiconti del
Circolo Matematico di Palermo for their comments that help improve
the manuscript.



\begin{thebibliography}{}
\footnotesize{

\bibitem{Amitsur}Amitsur, S.A. and Levitzki, J. \textit{Minimal identities for algebras.}
Proceeding of American Mathematical Society {\bf 1}, 449-463 (1950).

\bibitem{Azevedo1} Azevedo, S.S. \textit{Graded identities for the matrix algebra of order
n over an infinite field.} Communications in Algebra \textbf{30
(12)}, 5849-5860 (2002).

\bibitem{Azevedo2} Azevedo, S.S. \textit{A basis for $\mathbb{Z}$-graded identities of matrices
over infinite fields.} Serdica Mathematical Journal \textbf{29 (2)},
149-158 (2003).

\bibitem{Drensky2} Bahturin, Y. and Drensky, V.
\textit{Graded polynomial identities of matrices.} Linear Algebra
and its Applications {\bf 357 (1-3)}, 15-34 (2002).

\bibitem{Drensky3}Drensky, V. \textit{Free algebras and PI algebras.}
Graduate Course in Algebra, Springer-Verlag, Singapore (1999).

\bibitem{Drensky4} Drensky, V. and Formanek, E. \textit{Polynomial Identity Rings.}
Courses in Mathematics, Birkhauser,Basel (2004).

\bibitem{B} Brandão Júnior, A.P., Krasilnikov, A.N. and
Koshlukov, P.E. \textit{Graded central polynomials for the matrix
algebra of order two.}  Monatshefte für Mathematik \textbf{157},
247-256 (2009).

\bibitem{Junior} Brandão Júnior, A.P. \textit{Graded central polynomials for the algebra
$M_{n}(K)$.} Rendiconti del Circolo Matematico di Palermo
\textbf{57}, 265-278 (2008).

\bibitem{Plamen3} Colombo, J. and Koshlukov, P.E. \textit{
Central polynomials in the matrix algebra of order two.} Linear
Algebra and its Applications. \textbf{377}, 53-67 (2004).

\bibitem{Vincenzo} Di Vincenzo, O. M.\textit{On the graded identities of $M_{1,1}(E)$}. Israel
Journal of Mathematics {\bf 80}, 323-335 (1992).

\bibitem{Drensky1} Drensky, V. \textit{A minimal basis for the identities of a second-order
matrix algebra over a field of characteristic $0$}. Algebra and
Logic {\bf 20 (3)}, 188-194 (1981).

\bibitem{Fonseca} L.F.G. Fonseca. \textit{On the graded central polynomials for elementary gradings in
matrix algebras}. Rendiconti del Circolo Matematico di Palermo
\textbf{62 (2)}, 237-244 (2013).

\bibitem{Formanek}Formanek, E. \textit{Central polynomials for matrix
rings.} Journal of Algebra {\bf 23}, 129-132 (1972).

\bibitem{Kaplansky} Kaplansky, I. {\textit Rings with a polynomial identity.}
Bulletin of American Mathematical Society. {\bf 54}, 575-580 (1948).

\bibitem{Kaplansky2} Kaplansky, I. \textit{Problems in the theory of rings, Report
of a conference on linear algebras, June, 1956, in National Academy
of Science.} National Research Council, National Research Council
Publication \textbf{502}, 1-3 (1956).

\bibitem{Kemer1} Kemer, A.R. \textit{Varieties and $\mathbb{Z}_{2}$-graded algebras.}
Mathematics of the USSR Izvestya \textbf{25 (2)}, 359-374 (1985).

\bibitem{Kemer2} Kemer, A.R. \textit{Finite basis property of identities of associative
algebras.} Algebra and Logic \textbf{87 (5)}, 362-397 (1987).

\bibitem{Plamen1} Koshlukov, P.E. \textit{Basis of the identities of the matrix algebra of order two
over a field of characteristic $\neq 2$}. Journal of Algebra {\bf
241 (1)}, 410-434 (2001).

\bibitem{Okhitin} Okhitin, S. \textit{Central polynomials of the algebra of second
order matrices.} Moscow University Mathematics Bulletin \textbf{43
(4)}, 49-51 (1988).

\bibitem{Procesi} Procesi, C. \textit{Rings with polynomial identities}. Pure and applied
mathematics {\bf 17}, M. Dekker (1973).

\bibitem{Razmyslov2} Razmyslov, Yu. P. \textit{Finite basing of the identities of a matrix
algebra of second order over a field of characteristic zero.}
Algebra and Logic {\bf 12}, 47-63 (1973).

\bibitem{Razmyslov3} Razmyslov, Yu. P. \textit{On a problem of Kaplansky.}
Mathematics of the USSR-Izvestiya {\bf 7 (3)}, 479-496 (1973).

\bibitem{Silva}Silva, D.D.P.S. \textit{On the graded identities for elementary
gradings in matrix algebras over infinite fields.} Linear Algebra
and its Applications \textbf{439 (5)}, 1530-1537 (2013).

\bibitem{Specht} Specht, W.\textit{Gesetze in Ringen.} Mathematische Zeitschrift
\textbf{52}, 557-589 (1950).

\bibitem{Vasilovsky 1}Vasilovsky, S. Yu.\textit{$\mathbb{Z}_{n}$-graded polynomial identities of the full
matrix algebra of order n.} Proceedings of American Mathematical
Society {\bf 127 (12)}, 3517-3524 (1999).

\bibitem{Vasilovsky 2}Vasilovsky, S. Yu.\textit{$\mathbb{Z}$-graded polynomial identities of the full
matrix algebra}. Communications in Algebra {\bf 26, 2}, 601-612
(1998).}

\end{thebibliography}
\end{document}